\documentclass[a4paper,12pt]{article}
\usepackage{latexsym}
\usepackage{amsmath}
\usepackage{amssymb}
\usepackage{enumerate}
\usepackage{tikz}
\usepackage{color}

\usepackage{theorem}
\newtheorem{theorem}{Theorem}[section]
\newtheorem{proposition}[theorem]{Proposition}
\newtheorem{lemma}[theorem]{Lemma}
\newtheorem{corollary}[theorem]{Corollary}

\theorembodyfont{\rmfamily}
\newtheorem{proof}{\textmd{\textit{Proof.}}}

\newtheorem{remark}[theorem]{Remark}

\newtheorem{definition}[theorem]{Definition}

\newtheorem{acknowledgement}{\textmd{\textit{Acknowledgements.}}}

\makeatletter

\@addtoreset{equation}{section}
\makeatother

\newcommand{\qedd}{\hfill \Box}
\newcommand{\ve}{\varepsilon}

\newcommand{\B}{\ensuremath{\mathbb{B}}}

\newcommand{\R}{\ensuremath{\mathbb{R}}}

\newcommand{\Sph}{\ensuremath{\mathbb{S}}}

\newcommand{\bb}{\ensuremath{\mathbf{b}}}

\def\vol{\mathop{\mathrm{vol}}\nolimits}

\title{On the existence of convex functions on Finsler manifolds
\footnote{
Mathematics Subject Classification (2010)\,:\,53C60, 53C22.}
}
\author{Sorin V. SABAU and Pattrawut CHANSANGIAM}
\date{}
\begin{document}


\maketitle

\begin{abstract}
We show that  a non-compact (forward) complete Finsler manifold whose Holmes-Thompson volume is finite admits no non-trivial convex functions. We apply this result to some Finsler manifolds whose Busemann function is convex. 
\end{abstract}


\section{Introduction}

\quad Finsler manifolds are a natural generalization of Riemannian ones in the sense that the metric depends not only on the point, but on the direction as well. This generalization implies the non-reversibility of geodesics, the difficulty of defining angles and many other particular features that distinguish them from Riemannian manifolds. Even though classical Finsler geometry was mainly concerned with the local aspects of the theory, recently a great deal of effort was made to obtain global results in the geometry of Finsler manifolds (\cite{BCS}, \cite{Oh}, \cite{Sa}, \cite{ST} and many others).

In a previous paper \cite{SaS}, by extending the results in \cite{GS}, we have studied the geometry and topology of Finsler manifolds that admit convex functions, showing that such manifolds are subject to some topological restrictions. 
 We recall that a function $f:(M,F)\to \R$, defined on a (forward) complete Finsler manifold $(M,F)$, is called {\it convex} if and only if along every geodesic $\gamma:[a,b]\to M$, the composed function  $\varphi:=f\circ\gamma:[a,b]\to \R$ is convex, that is
\begin{equation}
f\circ\gamma[(1-\lambda)a+\lambda b]\leq (1-\lambda) f\circ \gamma (a) +\lambda f\circ\gamma (b), \quad 0\leq \lambda \leq 1.
\end{equation}
If the above inequality is strict for all geodesics $\gamma$, the function $f$ is called {\it strictly convex}, and if the equality holds good for all geodesics $\gamma$, then $f$ is called {\it linear}. A function $f:M\to \R$ is called {\it locally non-constant} if it is non-constant on any open subset $U$ of $M$, and {\it locally constant} otherwise. We are interested in locally non-constant convex functions on $M$.

It can be easily seen that any non-compact smooth manifold $M$ always admits a complete Riemannian or Finsler metric and a non-trivial smooth function which is convex with respect to this metric (see \cite{GS} for the Riemannian case and \cite{SaS} for the Finsler case).

On the other hand, it was shown by Yau (see \cite{Y}) that in the case of a non-compact manifold $M$, endowed with an {\it a priori} given complete Riemannian metric $g$, there is no non-trivial continuous convex function on $(M,g)$ if the Riemannian volume of $M$ is  finite.

In the present paper, we are going to generalize Yau's result to the case of Finsler manifolds, namely, {\it  if the non-compact manifold $M$ is endowed with an a priori  given (forward) complete Finsler metric, what are the conditions on $(M,F)$ for the existence of non-trivial convex functions}.

Recall that in the case of a Finsler manifold (we do not assume our Finsler norms to be absolute homogeneous), the induced volume is not unique as in the Riemannian case and hence several choices are available (see Section 3). The Busemann-Hausdorff and Holmes-Thompson volumes are the most well known ones.

Here is our main result.

\begin{theorem}\label{Thm: Main Theorem 1}
	Let $(M,F)$ be a (forward) complete non-compact Finsler manifold with finite Holmes-Thompson volume. 
	 Then any convex function $f:(M,F)\to \R$ must be constant.
\end{theorem}

Since all volume forms are bi-Lipschitz equivalent in the absolute homogeneous case (see for instance \cite{BBI}), then the result above holds good for any Finslerian volume, that is we have 
\begin{corollary}
	Let $(M,F)$ be an absolute homogeneous complete non-compact Finsler manifold endowed with a Finslerian volume measure.
	If the Finsler volume of $(M,F)$ is finite, then any convex function $f:(M,F)\to \R$ must be constant.
\end{corollary}

Our present results show that there are many topological restrictions on (forward) complete non-copmact Finsler manifolds with infinite Holmes-Thompson volume. Indeed, the topology of Finsler manifolds admitting convex functions was studied in detail in \cite{SaS}, hence the topological structure stated in the main three theorems in \cite{SaS} hold good for  (forward) complete non-copmact Finsler manifolds with infinite Holmes-Thompson volume.

\bigskip

Here is the structure of the paper.

In Section \ref{sec: Finsler manifolds} we recall the basic setting of Finsler manifolds $(M,F)$. In special, we present here the properties of the Riemannian volume of the indicatrix $SM$ and the invariance of this volume under the geodesic flow of $F$.

In Section \ref{sec: Finslerian volumes} we introduce the Busemann-Hausdorff and the Holmes-Thompson volumes of a Finsler manifold $(M,F)$, respectively, and point out the relation with the volume of the indicatrix. In particular, if the Holmes-Thompson volume of $(M,F)$ is finite, then the total measure of the indicatrix $SM$ is also finite (Proposition \ref{prop: finite HT vol and SM}).

Section \ref{sec: proof Thm 1.1} is where we prove Theorem \ref{Thm: Main Theorem 1} by making use of Lemmas \ref{lem: Key lemma}, \ref{lem: cyclic lim}, \ref{lem: L dense set}. In the proof of Lemma \ref{lem: L dense set} we use the Poincar\'e recurrence theorem (\cite{Po}).

Finally, in Section \ref{sec:Corollaries}, we apply Theorem \ref{Thm: Main Theorem 1} to the case of complete Berwald spaces of non-negative flag curvature and obtain that these kind of spaces must have infinite Holmes-Thompson volume (Corollary \ref{cor: Berwald vol infinite}). More generaly, a (forward) complete Finsler manifold of non-negative flag curvature whose Finsler-Minkowski norm $F_x$ is 2-uniformly smooth, at each point $x\in M$, must also have infinite Holmes-Thompson volume (Corollary \ref{cor:general case of infinte volume}).

\begin{acknowledgement}
	We are greatful 
	 to Prof. N. Innami and Prof. S. Ohta for their suggestions that have improved the quality of the paper. We also thank to Prof. H. Shimada for many useful discussions.
\end{acknowledgement}

\section{Finsler manifolds}\label{sec: Finsler manifolds}
Let $(M,F)$ be a (connected) $n$-dimensional Finsler manifold (see \cite{BCS} for basics of Finsler geometry).
The fundamental function $F$ of a Finsler structure $(M,F)$ determines and it is determined by the (tangent) {\it indicatrix}, or the total space of the unit tangent bundle of $F$, namely
\begin{equation*}
SM:=\{u\in TM:F(u)=1\}=\cup_{x\in M}S_xM
\end{equation*}
which is a smooth hypersurface of the tangent space $TM$. At each $x\in M$ we also have the {\it indicatrix at x}
\begin{equation*}
S_xM:=\{v\in T_xM \ |\  F(x,v)=1\}=\Sigma_F\cap T_xM
\end{equation*}
which is a smooth, closed, strictly convex hypersurface in
$T_xM$.

To give a Finsler structure $(M,F)$ is therefore equivalent to giving a smooth
hypersurface $SM\subset TM$ for which the canonical projection
$\pi:SM\to M$ is a surjective submersion and having the property
that for each $x\in M$, the $\pi$-fiber $S_xM=\pi^{-1}(x)$ is
strictly convex including the origin $O_x\in T_xM$.

Recall that the geodesic spray of $(M,F)$ is the vector field $S$, on the tangent space $TM$, given by
$$
S=y^i\frac{\partial}{\partial x^i}-2G^i(x,y)\frac{\partial}{\partial y^i},
$$
where $G^i:TM\to \R$ are the spray coefficients of $(M,F)$. For any $u=(x,y)\in TM$, the geodesic flow
of $(M,F)$ is the one parameter group of $S$, i.e.
$$
\phi:(-\ve,\ve)\times U\to TM,\quad u\mapsto \phi_t(u).
$$

The following result is well known.

\begin{lemma}[\cite{S}]\label{lem:form invariance}
	We have
\begin{enumerate}
	\item 	$$
	\frac{d}{dt}F(\phi_t(y))=dF(S_{\phi_t(y)})=0,
	$$
	that is $F(\phi_t(y))$ is constant.
	\item 	For any t, we have
	$$
	\frac{d}{dt}\Bigl[(\phi^*_t)\omega\Bigr]=\frac{1}{2}d\Bigl[
	(\phi^*_t)F^2
	\Bigr],
	$$
	where
	$
	\omega=g_{ij}(x,y)y^jdx^i$ is the Hilbert form of $(M,F)$.
\end{enumerate}



\end{lemma}

It follows
$$
(\phi^*_t)d\omega=d\omega.
$$

\section{Finslerian volumes}\label{sec: Finslerian volumes}
In order to fix notations, we recall that the {\it Euclidean volume form} in $\R^n$, with the coordinates $(x^1,x^2,\dots,x^n)$,  is the $n$-form
\[
dV_{\R^n}:=dx^1dx^2\dots dx^n,
\]
and the {\it Euclidean volume} of a bounded open set $\Omega\subset \R^n$ is given by
\begin{equation}\label{Euclidean vol}
\textrm{Vol}(\Omega)=\textrm{Vol}_{\R^n}(\Omega)=\int_\Omega dV_{\R^n}=\int_\Omega dx^1dx^2\dots dx^n.
\end{equation}


More generally, let us consider a Riemannian manifold $(M,g)$ with the {\it Riemannian volume form}
\[
dV_g:=\sqrt{g}dx^1dx^2\dots dx^n,
\]
and hence the {\it Riemannian volume} of $(M,g)$ can be computed as
\[
\textrm{Vol}(M,g)=\int_M dV_g=\int_M \sqrt{g}dx^1dx^2\dots dx^n=\int_M \theta^1\theta^2\dots\theta^n,
\]
where $\{\theta^1,\theta^2,\dots,\theta^n\}$ is a $g$-orthonormal co-frame on $M$.


We remark that this Riemannian volume is uniquely determined by the following two properties:
\begin{enumerate}
	\item The Riemannian volume in $\R^n$ is the standard Euclidean volume \eqref{Euclidean vol}.
	\item The volume is monotone with the metric.
\end{enumerate}

On the other hand, in the Finslerian case, this is not true anymore.
Indeed, even if we ask for the Finslerian volume to satisfy the same two properties above, the volume is not uniquely defined, but depends on the choice of a positive function on $M$. More precisely, a {\it volume form} $d\mu$ on an $n$-dimensional  Finsler manifold $(M,F)$ is a global defined, non-degenerate $n$-form on $M$ written in the local coordinates $(x^1,x^2,\dots,x^n)$ of $M$ as
\begin{equation}
d\mu=\sigma(x) dx^1\wedge\dots\wedge dx^n,
\end{equation}
where $\sigma$ is a positive function on $M$ (see \cite{BBI} for details in the absolute homogeneous case).

Depending on the choice of $\sigma$
several different volume forms are known: the Busemann volume, the Holmes-Thompson volume, etc.






The {\it Busemann-Hausdorff} volume form is defined as
\begin{equation}
dV_{BH}:=\sigma_{BH}(x) dx^1\wedge\dots\wedge dx^n,
\end{equation}
where
\begin{equation}
\sigma_{BH}(x):=\frac{\textrm{Vol}(\B^n(1))}{\textrm{Vol}(\mathcal B^n_xM)},
\end{equation}
here $\B^n(1)$ is the Euclidean unit $n$-ball, $\mathcal B_x^nM=\{y:F(x,y)=1\}$ is the Finslerian ball and Vol the canonical Euclidean volume.

The {\it Busemann-Hausdorff} volume of the Finsler manifold $(M,F)$ is defined by
$$
\textrm{vol}_{BH}(M,F)=\int_M dV_{BH}.
$$

Using the Brunn-Minkowski theory, Busemann showed in \cite{Bu1} that the Busemann-Hausdorff volume of an $n$-dimensional normed space equals its $n$-dimensional Hausdorff volume, hence the naming.

However, we point out that except for the case of absolute homogeneous Finsler manifolds, the Busemann-Hausdorff volume does not have the expected geometrical properties, and hence it is not suitable for the study of Finsler manifolds (see \cite{Pa} for a description of these properties and the main issues that appear; see also \cite{Cr} for the Berwald case when the Busemann-Hausdorff volume has some special properties).

\begin{remark}
	 Observe that the $n$-ball Euclidean volume is
		$$
		\textrm{Vol}(\mathbb B^n(1))=\frac{1}{n}\textrm{Vol}(\Sph^{n-1})=
		\frac{1}{n}\textrm{Vol}(\Sph^{n-2})\int_0^\pi \sin^{n-2}(t)dt.
		$$
	
\end{remark}
\bigskip

Another volume form naturally associated to a Finsler structure is the {\it Holmes-Thompson} volume defined by
\begin{equation}
\sigma_{HT}(x):=\frac{\textrm{Vol}(\mathcal B_x^nM,{g_x})}
{\textrm{Vol}(\mathbb B^n(1))}=
\frac{1}{\textrm{Vol}(\mathbb B^n(1))}\int_{\mathcal B_x^nM}\det g_{ij}(x,y)dy,
\end{equation}
and the  {\it Holmes-Thompson} volume of the Finsler manifold $(M,F)$ is defined as
$$
\textrm{vol}_{HT}(M,F)=\int_M dV_{HT}.
$$

This volume was introduced by Holmes and Thompson in \cite{HT1} from geometrical reasons as the dual functor of Busemann-Hausdorff volume. It has better geometrical properties than the Busemann-Hausdorff volume and hence we consider it appropiate for the study of Finsler manifolds.

\begin{remark}
	\begin{enumerate}
		\item 	If $(M,F)$ is an absolute homogeneous Finsler manifold, then
		 the Busemann-Hausdorff volume is a Hausdorff measure of $M$, and we have
		$$
		\textrm{vol}_{HT}(M,F)\leq \textrm{vol}_{BH}(M,F) .
		$$
		(see \cite{Du}).
		
		\item If $(M,F)$ is not absolute homogeneous, then the inequality above is not true anymore. Indeed, for instance let $(M,F=\alpha+\beta)$ be a Randers space. Then, one can easily see that
		$$
		\textrm{vol}_{BH}(M,F)=\int_M (1-b^2(x))dV_\alpha\leq \textrm{vol}(M,\alpha)=\textrm{vol}_{HT}(M,F),
		$$
		where $b^2(x)=a_{ij}(x)b^ib^j$, and $\textrm{vol}(M,\alpha)$ is the Riemannian volume of $M$ (see \cite{S}).

		In the case of a smooth surface endowed with a positive defined slope metric $(M,F-\frac{\alpha^2}{\alpha-\beta})$, we have
		$$
				\textrm{vol}_{BH}(M,F)< \textrm{vol}_{HT}(M,F)< \textrm{vol}(M,\alpha),
		$$
		where $\alpha$ and $\beta$ are the same as above (see \cite{PCS1}).
		
		More generally, in the case of an $(\alpha,\beta)$, one can compute explicitely the Finslerian volume in terms of the Riemannian volume (see \cite{BCH}). Indeed, if $(M,F(\alpha,\beta))$ is an $(\alpha,\beta)$ metric on an $n$-dimensional manifold $M$, one denotes
	\begin{equation}\label{f,g}
	\begin{split}
	&f(b):=\frac{\int_0^\pi \sin^{n-2}(t)dt}{\int_{0}^\pi\frac{\sin^{n-2}(t)}{\phi(b\cos(t))^n}dt}\\
	& g(b):=\frac{\int_0^\pi\sin^{n-2}(t) T(b\cos t)dt }{\int_0^\pi \sin^{n-2}(t)dt},
	\end{split}
	\end{equation}
where $F=\alpha \phi(s)$, $s=\beta \slash \alpha$, and
$$
T(s):=\phi(\phi-s\phi')^{n-2}[(\phi-s\phi')+(b^2-s^2)\phi''].
$$

Then the Busemann-Hausdorff and Holmes-Thompson volume forms are given by
$$
dV_{BH}=f(b)dV_\alpha,\textrm{ and } dV_{HT}=g(b)dV_\alpha,
$$
respectively, where $f$ and $g$ are given by \eqref{f,g}.
		
It is remarkable that if the function $T(s)-1$ is an odd function of $s$, then $dV_{HT}=dV_\alpha$. This is the case of Randers metrics.
	\end{enumerate}

\end{remark}


We will consider now the volume induced by the Hilbert form
$$
\omega:=g_{ij}(x,y)y^jdx^i=p_idx^i
$$
of the Finsler manifold $(M,F)$.

It follows
$$
d\omega=\frac{\partial g_{jk}}{\partial x^i}y^kdx^i\wedge dx^j-g_{ij}dx^i\wedge dy^j,
$$
and hence, we have
$$
(d\omega)^n=d\omega\wedge \dots \wedge d\omega =(-1)^\frac{n(n+1)}{2}n!\det g_{ij}(x,y)\  dx^1\wedge \dots dx^n\wedge dy^1\wedge \dots dy^n.
$$

The Hilbert form $\omega$ induces a volume form on $TM\setminus \{0\}$ defined by
$$
dV_\omega:=(-1)^\frac{n(n+1)}{2}\frac{1}{n!}(d\omega)^n=\det |g(x,y)|\ dx\wedge dy,
$$
where $\det |g(x,y)|$ is the determinant of the matrix $g_{ij}(x,y)$.

Observe that the volume of $(M,F)$ defined as
$$
\vol_\omega(M,F):=\frac{1}{\textrm{Vol}(\mathbb B^n(1))}\int_{BM}dV_\omega=\vol_{HT}(M,F),
$$
where $BM:=\{(x,y)\in TM\ : \ F(x,y)<1\}\subset TM$, is in fact the same as the HT-volume of the Finsler manifold $(M,F)$.

The following lemma is elementary.

\begin{lemma}\label{lem:HT vol and SM sympl vol}
	The following formula holds good
	\begin{equation}
\vol_{HT}(M,F)=
\frac{1}{(2n-1) {\textrm　{Vol}} (\mathbb B^n(1))}\int_{SM}dV_\omega.
	\end{equation}
\end{lemma}
     Indeed, it is useful to observe first that
  \begin{equation}\label{ball and indic volumes}
 \int_{B_xM}dV_\omega=\frac{1}{(2n-1)}\int_{S_xM}dV_\omega.
  \end{equation}
     
   To see this, it is easy to see that, due to homogeneity, we can identify $T_xM\setminus \{0\}$ with $(0,\infty)\times S_xM$, by
   $$
   y\mapsto (F(y),\frac{y}{F(y)}).
   $$
     It follows that
     $$
     G=(dt)^2\oplus t^2 \hat{G},
     $$
     where $t\in (0,\infty)$, $G$ is the Riemannian metric of $T_xM\setminus \{0\}$, that is the Sasakian metric, and $\hat{G}$ is the restriction of $G$ to $S_xM$.
     
     Then
     $$
     \det |G|=t^{2n-2}\det |\hat{G}|,
     $$
     and hence
     $$
     \int_{B_xM}dV_\omega=\int_{B_xM}\det|G|dy=
     \int_0^1 t^{2n-2}dt\int_{S_xM}dV_\omega,
     $$
     therefore \eqref{ball and indic volumes} follows. By integrating this formula over $M$ we get the formula in Lemma \ref{lem:HT vol and SM sympl vol}.
     
     From Lemma \ref{lem:HT vol and SM sympl vol} we obtain
     \begin{proposition}\label{prop: finite HT vol and SM}
     	Let $(M,F)$ be a Finsler metric whose Holmes-Thompson volume is finite. Then
     		the symplectic volume $\vol_\omega(SM)=\int_{SM}dV_\omega$  of $SM$ is also finite.
     \end{proposition}
    
     We recall for later use the folowing Liouville-type theorem.
     \begin{theorem}
     	The volume form $dV_\omega$ is invariant under the geodesic flow of $(M,G)$.
     \end{theorem}
     
     The proof is trivial taking into account  Lemma \ref{lem:form invariance}.

\section{The proof of Theorem 1.1}\label{sec: proof Thm 1.1}

In the following, let $(M,F)$ be a non-compact (forward) complete Finsler manifold with bounded Holmes-Thompson volume, and let $f:(M,F)\to \R$ be a convex function on $M$. We denote again by $\phi$ the geodesic flow of $F$ on $SM$.

	
	
	Taking into account that a convex function cannot be bounded, from the convexity of $f$ it is elementary to see that
	
	\begin{lemma}\label{lem: Key lemma}
 If  $\gamma:[0,\infty)\to M$ is any $F$-geodesic
on $M$ such that $\lim_{i\to \infty}\gamma(t_i)=\gamma(0)$ for some divergent numerical sequence $\{t_i\}$, $\lim_{i\to \infty}t_i=\infty$, then $f\circ \gamma:[0,\infty)\to \R$ must be constant.
\end{lemma}

Moreover, we have

\begin{lemma}\label{lem: cyclic lim}
	For any open set $U\subset SM$, there is an infinite sequence $t_i$, $\lim_{i\to \infty}t_i=\infty$ such that
	$$
	\phi_{t_i}(U)\cap U\neq \emptyset,\quad \textrm{ for all } t_i,
	$$
	where $\phi_t$ is the one parameter group generated by the geodesic flow of $(M,g)$.
\end{lemma}

Indeed, if we assume the contrary, then there are infinitely many pairwise disjoint open sets with equal measure, which contradicts the fact that $SM$ has finite symplectic volume.

\begin{lemma}\label{lem: L dense set}
	The set of points
	$$
	L:=\{u\in SM\ :\ \lim_{t_i\to \infty}\phi_{t_i}(u)=u,\textrm{ for some sequence } t_i\to \infty
	\}
	$$
	is dense in SM.
\end{lemma}

\begin{proof}
 The result follows from the more general Poincare recurrence theorem (\cite{Po}), that is,  {\it the set of recurrent points, of a measure preserving flow on a measure space with bounded measure, is a full measure set}.


Finally, observe that, being of full measure, the set of recurrent vectors must be in fact dense subset of $SM$. The proof is complete.
	$\qedd$
\end{proof}

\begin{remark}
In the proof above we have used the fact that a full measure subset $X$, of a space $E$ with measure, is dense in $E$. Observe that the inverse is not true because one can easily construct examples of dense subsets that are not of full measure.
  \end{remark}

Now the main theorem can be proved.

\begin{proof}[Proof of the Theorem \ref{Thm: Main Theorem 1}]
	Consider any point $u=(p,v)\in SM$. Since $L$ is dense (see Lemma \ref{lem: L dense set}), there always exists a sequence of points $u_i\in SM$ converging to $u$, i.e. $\lim_{i\to \infty}u_i=u$.
	
	Let $\gamma_u$ and $\gamma_{u_i}$ be the geodesics on $(M,g)$ determined by $u$ and $u_i$.
	
	Observe that Lemma \ref{lem: Key lemma} implies that $f\circ\gamma_{u_i}$ must be constant for any $i$. By continuity it follows that $f\circ\gamma$ must also be constant.
	
	Therefore, $f$ is locally constant, thus must be constant on $M$.
	$\qedd$
\end{proof}

\section{Corollaries}\label{sec:Corollaries}

Recall that a function $f:(M,F)\to \R$ defined on a non-compact (forward) complete Finsler manifold is called convex if $f\circ\gamma:[0,1]\to \R$ is a convex function in the usual sense, for any $\gamma:[0,1]\to M$ Finsler geodesic. To be non-compact is a necessary condition for the existence of non-trivial convex functions. Indeed, it is trivial to see that if $M$ is compact, then $f$ must be bounded and hence constant.

Let $(M,F)$ be a forward complete boundaryless Finsler manifold. A unit speed globally minimizing geodesic $\gamma:[0,\infty)\to
M$ is called a {\it (forward) ray}. A ray $\gamma$ is called
{\it maximal} if it is not a proper sub-ray of another ray, i.e.
for any $\ve>0$ its extension to $[-\ve,\infty)$ is not a ray
anymore. Moreover, let us assume that $(M,F)$ is bi-complete, i.e. forward and backward complete.
A Finslerian unit speed globally minimizing geodesic $\gamma:\R\to
M$ is called a {\it  straight line}. We point out that, even though for defining rays and straight lines  we not need any completeness hypothesis, without completeness, introducing rays and straight lines would be meaningless.

Let $(M,F)$ be a forward complete boundaryless non-compact Finsler manifold (see \cite{BCS}, \cite{S} for details on the completeness of Finsler manifolds).
In Riemannian geometry, the forward and backward completeness are equivalent, hence the words ``forward'' and ``backward'' are superfluous, but in Finsler geometry these are not equivalent anymore.

\begin{definition}\label{def:Bus_fct}
	If $\gamma:[0,\infty)\to M$ is a ray in a forward complete boundaryless non-compact Finsler manifold $(M,F)$, then the function
	\begin{equation}\label{Bus_fnc_def}
	\bb_\gamma:M\to\R,\quad
	\bb_\gamma(x):=\lim_{t\to\infty}\{t-d(x,\gamma(t))\}
	\end{equation}
	is called {\it the Busemann function with respect to
		$\gamma$}, where $d$ is the Finsler distance function.
\end{definition}

See  \cite{Oh} and \cite{Sa} for basic results on Busemann function for Finsler manifolds.



		
		

It is known that the Busemann function of a non-compact complete Riemannian manifold of non-negative sectional curvature is convex. However, in the Finslerian case, due to the different behaviour of geodesics and the dependence of the metric on direction, bounded conditions on the flag curvature are not enough to assure the convexity of the Busemann function $\bb_\gamma$.

The case of Berwald spaces is well understood. Indeed, the Busemann function of any Berwald space of non-negative flag curvature is convex (see \cite{Oh2}, \cite{Oh}, \cite{Kell}).

From our Main Theorem it follows
\begin{corollary}\label{cor: Berwald vol infinite}
The Holmes-Thompson volume of a Berwald space of non-negative flag curvature is infinite.
\end{corollary}

\begin{remark}
	If $(M,F)$ is a Berwald space of non-negative flag curvature, then Corollary \ref{cor: Berwald vol infinite} can be also proved exactly as in the Riemannian case (see \cite{Wu} for an elementary  proof of the Riemannian case). Indeed, the specific features of Berwald spaces, like the reversibility of geodesics, the vanishing of the tangent curvature and the formula for the second variation of the arc length (see \cite{BCS} or \cite{S}), make the Riemannian arguments working in the Berwald case.
	
\end{remark}

\begin{remark}
Let us also observe that in the Berwald case, the volume conditions obtained above also holds good for the Busemann-Hausdorff volume. Even thought we have pointed out that the Busemann-Hausdorff volume is not quite suitable for the study of arbitrary Finsler manifolds, in the Berwald case it has some special properties that make it more usefull that in the general case. Indeed, if $(M,F)$ is a Berwald space, then by averaging over the indicatrices, one can obtain a Riemannian metric (actually several Riemannian metrics depending on the averaging formula, see \cite{Cr}) whose volume is proportional with the Busemann-Hausdorff volume. The details follow easily.
\end{remark}

Observe that  the papers \cite{Oh2}, \cite{Oh}, \cite{Kell}  link the notion of uniform smoothness with the convexity of Busemann function. Indeed, the essential result is that if $(M,F)$ is a non-compact connected (forward) complet Finsler manifold such that
\begin{enumerate}
\item it is of non-negative flag curvature,
\item for all $x\in M$, the Finsler-Minkowski  norms $F_x$ are 2-uniformly smooth,
\end{enumerate}
then for any reversible ray $\gamma:[0,\infty)\to M$, the Busemann function $\bb_\gamma$ is convex (see \cite{Kell} Lemma 3.11, Corollary 3.12).

By combinig this result with Theorem \ref{Thm: Main Theorem 1} it results
\begin{corollary}\label{cor:general case of infinte volume}
 The Finsler manifolds with the properties 1, 2 above must have infinite Holmes-Thompson volume.
\end{corollary}




\bigskip

S. V. SABAU


\noindent School of Biological Sciences,
Department of Biology,\\
Tokai University,
Sapporo 005\,--\,8600,
Japan

\medskip
{\tt sorin@tokai.ac.jp}

\bigskip

P. CHANSANGIAM


\noindent Faculty of Science,
Department of Mathematics,\\
King Mongkut's Institute of Technology Ladkrabang,
Bangkok 10520,
Thailand

\medskip
{\tt pattrawut.ch@kmitl.ac.th}
\end{document}